\documentclass[10pt,reqno]{amsart}

\usepackage{amsthm, mathrsfs, amsmath, amstext, amsxtra, amsfonts, dsfont, amssymb, bm}
\usepackage{xcolor}
\usepackage{lmodern}
\usepackage[colorlinks, linkcolor=red, citecolor=blue, urlcolor=blue, pagebackref, hypertexnames=false]{hyperref}

\newcommand{\R}{\mathbb R}
\newcommand{\C}{\mathbb C}
\newcommand{\Ac}{\mathcal A}

\newcommand{\Gc}{\mathcal G}

\newcommand{\Mcal}{\mathcal M}
\newcommand{\eps}{\epsilon}

\newcommand{\n}{\bm n}
\newcommand{\scal}[1]{\left\langle #1 \right\rangle} 

\newcommand{\defendproof}{\hfill $\Box$} 

\newtheorem{theorem}{Theorem}[section]
\newtheorem{lemma}[theorem]{Lemma} 
\newtheorem{proposition}[theorem]{Proposition}
 
\theoremstyle{definition}
\newtheorem{definition}[theorem]{Definition}
\newtheorem{remark}[theorem]{Remark}

\title[Instability radial standing waves]{On instability of radial standing waves for the nonlinear Schr\"odinger equation with inverse-square potential} 

\author{Van Duong Dinh}
\address{
Institut de Mathematiques de Toulouse UMR5219, Universit\'e de Toulouse CNRS, 31062 Toulouse Cedex 9, France
and 
Department of Mathematics, HCMC University of Pedagogy, 280 An Duong Vuong, Ho Chi Minh, Vietnam
}
\email{contact@duongdinh.com}


\begin{document}
	
	\begin{abstract}
		We show the strong instability of radial ground state standing waves for the focusing $L^2$-supercritical nonlinear Schr\"odinger equation with inverse-square potential
		\[
		i\partial_t u + \Delta u + c|x|^{-2} u = - |u|^{\alpha} u, \quad (t,x)\in \R \times \R^d,
		\]
		where $d\geq 3$, $u: \R \times \R^d \rightarrow \C$, $c\ne 0$ satisfies $c<\lambda(d):=\left(\frac{d-2}{2}\right)^2$ and $\frac{4}{d} <\alpha<\frac{4}{d-2}$. This result extends a recent result of Bensouilah-Dinh-Zhu [{\it On stability and instability of standing waves for the nonlinear Schr\"odinger equation with inverse-square potential}, \url{arXiv:1805.01245}] where the stability and instability of standing waves were shown in the $L^2$-subcritical and $L^2$-critical cases.
	\end{abstract}
	
	\maketitle
	
	\section{Introduction}
	\setcounter{equation}{0}
	In the last decade, there has been a great deal of interest in studying the nonlinear Schr\"odinger equation with inverse-square potential, namely
	\begin{align}
	i\partial_t u + \Delta u + c|x|^{-2} u = \mu |u|^\alpha u, \quad (t,x) \in \R \times \R^d, \label{NLS inverse square introduction}
	\end{align}
	where $d\geq 3$, $u: \R \times \R^d \rightarrow \C$, $c \ne 0$ satisfies $c<\lambda(d):=\left(\frac{d-2}{2}\right)^2$, $\mu \in \R$ and $\alpha>0$. The nonlinear Schr\"odinger equation $(\ref{NLS inverse square introduction})$ appears in a variety of physical settings, such as quantum field equations or black hole solutions of the Einstein's equations (see e.g. \cite{Case, CamEpeFanCan, KalSchWalWus}) and quantum gas theory (see e.g. \cite{AstrakharchikMalomed, SakaguchiMalomed-11, SakaguchiMalomed-13}). The mathematical interest in the nonlinear Schr\"odinger equation with inverse-square potential comes from the fact that the potential is homogeneous of degree $-2$ and thus scales exactly the same as the Laplacian. Recently, the equation $(\ref{NLS inverse square introduction})$ has been intensively studied (see e.g. \cite{Bensouilah, BensouilahDinh, BensouilahDinhZhu, BurPlaStaZad, CsoboGenoud, Dinh-inverse, KilMiaVisZhaZhe-energy, KillipMurphyVisanZheng, OkazawaSuzukiYokota, TracZogra, ZhangZheng} and references therein). 
	
	In this paper, we consider the $L^2$-supercritical nonlinear Schr\"odinger equation with inverse-square potential, namely
	\begin{align}
	\left\{
	\begin{array}{rcl}
	i\partial_t u + \Delta u + c|x|^{-2} u &=& - |u|^{\alpha} u, \quad (t,x)\in \R \times \R^d, \\
	u(0)&=& u_0 \in H^1,
	\end{array} 
	\right. \label{inverse square NLS}
	\end{align}
	where $d\geq 3$, $u: \R \times \R^d \rightarrow \C$, $u_0:\R^d \rightarrow \C$, $c\ne 0$ satisfies $c<\lambda(d)$ and $\frac{4}{d} <\alpha<\frac{4}{d-2}$. 
	
	The main purpose of this paper is to study the instability of radial ground state standing waves for \eqref{inverse square NLS}. Before stating our result, let us recall known results related to the stability and instability of standing waves for the nonlinear Schr\"odinger-like equations. The stability of standing waves for the classical nonlinear Schr\"odinger equation (i.e. $c=0$ in $(\ref{inverse square NLS})$) is widely pursued by physicists and mathematicians (see e.g. \cite{F} for reviews). To our knowledge, the first work addressed the orbital stability of standing waves for the classical NLS belongs to Cazenave-Lions \cite{CazenaveLions} via the concentration-compactness principle. Later, Weinstein in \cite{Weinstein85, Weinstein86} gave another approach to prove the orbital stability of standing waves for the classical NLS. Afterwards, Grillakis-Shatah-Strauss in \cite{GrillakisShatahStrauss87, GrillakisShatahStrauss90} gave a criterion based on a form of coercivity for the action functional (see $(\ref{action functional})$) to prove the stability of standing waves for a Hamiltonian system which is invariant under a one-parameter group of operators. Since then, a lot of results on the orbital stability of standing waves for nonlinear dispersive equations were obtained. For the nonlinear Schr\"odinger equation with a harmonic potential, Zhang \cite{Zhang} succeeded in obtaining the orbital stability of standing waves by the weighted compactness lemma. Recently, the orbital stability phenomenon was proved for the fractional nonlinear Schr\"{o}dinger equation by establishing the profile decomposition for bounded sequences in $H^s$ (see e.g. \cite{PengShi, ZhangZhu}). The instability of standing waves for the classical NLS was first studied by Berestycki-Cazenave \cite{BerestyckiCazenave} (see also \cite{Cazenave}). Later, Le Coz in \cite{LeCoz08} gave an alternative, simple proof of the classical result of Berestycki-Cazenave. The key point is to establish the finite time blow-up by using the variational characterization of the ground states as minimizers of the action functional and the virial identity. For the Schr\"odinger equations with more general nonlinearities, this method does not work due the the lack of virial identities. In such cases, one may use a powerful tool of Grillakis-Shatah-Strauss \cite{GrillakisShatahStrauss87, GrillakisShatahStrauss90} to derive the instability of standing waves.  
	
	Recently, the authors in \cite{BensouilahDinhZhu} succeeded, using a profile decomposition theorem  proved by the first author \cite{Bensouilah}, to establish the stability of standing waves for \eqref{inverse square NLS} in the $L^2$-subcritical regime and the instability by blow-up in the $L^2$-critical regime. The main goal here is to extend these results to the $L^2$-supercritical case but only for radial ground state standing waves.
	
	Throughout this paper, we call a standing wave a solution of $(\ref{inverse square NLS})$ of the form $e^{i\omega t} \phi_\omega$, where $\omega \in \R$ is a frequency and $\phi_\omega \in H^1$ is a nontrivial solution to the elliptic equation
	\begin{align}
	-\Delta \phi_\omega + \omega \phi_\omega - c|x|^{-2} \phi_\omega - |\phi_\omega|^\alpha \phi_\omega =0. \label{elliptic equation}
	\end{align}
	
	Note that the existence of positive radial solutions to the elliptic equation
	\[
	-\Delta \phi + \phi - c|x|^{-2} \phi - |\phi|^\alpha \phi =0
	\]
	was shown in \cite[Theorem 3.1]{KillipMurphyVisanZheng} and \cite[Theorem 4.1]{Dinh-inverse}. By setting $\phi_\omega(x):= \left(\sqrt{\omega} \right)^{\frac{2}{\alpha}} \phi(\sqrt{\omega}x)$, it is easy to see that $\phi_\omega$ is a solution of $(\ref{elliptic equation})$. This shows the existence of positive radial solutions to $(\ref{elliptic equation})$.
	
	Note also that $(\ref{elliptic equation})$ can be written as $S'_\omega(\phi_\omega)=0$, where
	\begin{align}
	\begin{aligned}
	S_\omega(v) &:= E(v) + \frac{\omega}{2} \|v\|^2_{L^2} \\
	&\mathrel{\phantom{:}}= \frac{1}{2} \|v\|^2_{\dot{H}^1_c} + \frac{\omega}{2} \|v\|^2_{L^2} -\frac{1}{\alpha+2} \|v\|^{\alpha+2}_{L^{\alpha+2}}
	\end{aligned}
	\label{action functional}
	\end{align}
	is the action functional. Here 
	\begin{align}
	\|v\|^2_{\dot{H}^1_c} := \|\nabla v\|^2_{L^2} - c\||x|^{-1} v\|^2_{L^2} \label{hardy functional}
	\end{align}
	is the Hardy functional. 
	
	We denote the set of non-trivial radial solutions of $(\ref{elliptic equation})$ by
	\[
	\Ac_{\text{rad},\omega}:= \left\{ v \in H^1_{\text{rad}} \backslash \{0\} \ : \ S'_\omega(v) =0 \right\},
	\]
	where $H^1_{\text{rad}}$ is the space of radial $H^1$ functions.
	
	\begin{definition} [Radial ground states] \label{definition radial ground state}
		A function $\phi \in \Ac_{\text{rad},\omega}$ is called {\bf a radial ground state} for $(\ref{elliptic equation})$ if it is a minimizer of $S_\omega$ over the set $\Ac_{\text{rad},\omega}$. The set of radial ground states is denoted by $\Gc_{\text{rad},\omega}$. In particular,
		\[
		\Gc_{\text{rad},\omega} = \left\{ \phi \in \Ac_{\text{rad},\omega} \ : \ S_\omega(\phi) \leq S_\omega(v), \ \forall v \in \Ac_{\text{rad},\omega} \right\}.
		\]
	\end{definition}
	We have the following result on the existence of radial ground states for $(\ref{elliptic equation})$.
	\begin{proposition} \label{proposition existence radial ground states}
		Let $d\geq 3$, $c\ne 0$ be such that $c<\lambda(d)$, $\frac{4}{d}<\alpha<\frac{4}{d-2}$ and $\omega>0$. Then the set $\mathcal{G}_{\emph{rad},\omega}$ is not empty, and it is characterized by
		\[
		\mathcal{G}_{\emph{rad},\omega} = \left\{ v \in H^1_{\emph{rad}} \backslash \{0\}, \ : \ S_\omega(v) = d(\emph{rad},\omega), \ K_\omega(v)=0 \right\},
		\]
		where
		\[
		K_\omega(v):= \left. \partial_\lambda S_\omega(\lambda v) \right|_{\lambda=1} = \|v\|^2_{\dot{H}^1_c} + \omega \|v\|^2_{L^2} - \|v\|^{\alpha+2}_{L^{\alpha+2}}
		\]
		is the Nehari functional and
		\begin{align}
		d(\emph{rad},\omega):= \inf \left\{ S_\omega(v) \ : \ v \in H^1_{\text{rad}} \backslash \{0\}, \ K_\omega (v) =0 \right\}. \label{minimizing problem}
		\end{align}
	\end{proposition}
	We refer the reader to Section $\ref{section existence ground state}$ for the proof of the above result.
	\begin{remark}
		Recently, Fukaya-Ohta in \cite{FukayaOhta} studied the instability of standing waves for the nonlinear Schr\"odinger equation with an attractive inverse power potential, namely
		\[
		i\partial_t u + \Delta u + \gamma |x|^{-\alpha} u = -|u|^{p-1} u,
		\]
		where $\gamma>0, 0<\alpha<\min\{2, d\}$ and $\frac{4}{d}<p-1<\frac{4}{d-2}$ if $d\geq 3$ and $\frac{4}{d}<p-1<\infty$ if $d=1$ or $d=2$. The potential $V(x) = \gamma |x|^{-\alpha}$ belongs to $L^r(\R^d) + L^\infty(\R^d)$ for some $r> \min \{1,d/2\}$. This special property allows them to use the weak continuity of the potential energy (see e.g. \cite[Theorem 11.4]{LiebLoss}) to prove the existence of non-radial ground states. In our case, the inverse-square potential $V(x)=c |x|^{-2}$ does not belong to $L^{\frac{d}{2}}(\R^d) + L^\infty(\R^d)$, so the weak continuity of potential energy is not applicable to our potential. At the moment, we do not know how to show the existence of non-radial ground states for $(\ref{elliptic equation})$. We hope to consider this problem in a future work.
	\end{remark}
	
	Let us now recall the definition of the strong instability.
	\begin{definition}[Strong instability] 
		We say that the standing wave $e^{i\omega t} \phi_\omega$ is strongly unstable if for any $\eps>0$, there exists $u_0 \in H^1$ such that $\|u_0 - \phi_\omega\|_{H^1} <\eps$ and the solution $u(t)$ of $(\ref{inverse square NLS})$ with initial data $u_0$ blows up in finite time.
	\end{definition}
	
	Our main result of this paper is the following:
	\begin{theorem} \label{theorem instability}
		Let $d\geq 3$, $c\ne 0$ be such that $c<\lambda(d)$, $\frac{4}{d}<
		\alpha<\frac{4}{d-2}$, $\omega>0$ and $\phi_\omega \in \Gc_{\emph{rad},\omega}$. Then the standing wave solution $e^{i\omega t} \phi_\omega$ of $(\ref{inverse square NLS})$ is strongly unstable. 
	\end{theorem}
	
	To our knowledge, the usual strategy to show the strong instability of standing waves is to use the characterization of ground states combined with the virial identity. However, in the presence of the inverse-square potential, the existence of ground states is well-known. However, the regularity as well as the decay of ground states are not yet known. Therefore, it is not known that the ground states $\phi_\omega$ belongs to the weighted space $\Sigma: = H^1 \cap L^2(|x|^2 dx)$ in order to apply the virial identity. This is a reason why we only consider the instability of radial ground state standing waves in this paper. If one can show that $\phi_\omega \in \Sigma$, then one can study the instability of non-radial ground state standing waves.
	
	The proof of Theorem $\ref{theorem instability}$ is based on the characterization of the radial ground states and the localized virial estimates. Thanks to the radial symmetry of the ground state, we are able to use the localized virial estimates derived by the second author in \cite{Dinh-inverse} to show the finite time blow-up. We refer the reader to Section $\ref{section instability}$ for more details.
	
	The rest of the paper is organized as follows. In Section $\ref{section existence ground state}$, we give the proof of the existence of radial ground states for $(\ref{elliptic equation})$ given in Proposition $\ref{proposition existence radial ground states}$. The proof of our main result-Theorem $\ref{theorem instability}$ will be given in Section $\ref{section instability}$. 
	
	\section{Existence of radial ground states} \label{section existence ground state}
	\setcounter{equation}{0}
	In this section, we give the proof the existence of radial ground states for $(\ref{elliptic equation})$ given in Proposition $\ref{proposition existence radial ground states}$. The proof of Proposition $\ref{proposition existence radial ground states}$ follows from several lemmas. Let us denote the $\omega$-Hardy functional by
	\[
	H_\omega(v):= \|v\|^2_{\dot{H}^1_c} + \omega \|v\|^2_{L^2}.
	\]
	Using the sharp Hardy inequality
	\[
	\lambda(d) \||x|^{-1} v\|^2_{L^2} \leq \|\nabla v\|^2_{L^2},
	\]
	we see that for $c<\lambda(d)$ and $\omega>0$ fixed, 
	\begin{align}
	H_\omega(v) \sim \|v\|^2_{H^1}. \label{equivalent norms}
	\end{align}
	We note that the action functional can be rewritten as
	\begin{align}
	S_\omega(v):= \frac{1}{2} K_\omega(v) + \frac{\alpha}{2(\alpha+2)} \|v\|^{\alpha+2}_{L^{\alpha+2}} = \frac{1}{\alpha+2} K_\omega(v) + \frac{\alpha}{2(\alpha+2)} H_\omega(v). \label{expressions S_omega}
	\end{align}

	Let us start with the following result.
	\begin{lemma} \label{lemma positivity d_omega}
		$d(\emph{rad},\omega)>0$.
	\end{lemma}
	
	\begin{proof}
		Let $v \in H^1_{\text{rad}} \backslash \{0\}$ be such that $K_\omega(v) =0$. By the Sobolev embedding, $(\ref{equivalent norms})$ and the fact $H_\omega(v) = \|v\|^{\alpha+2}_{L^{\alpha+2}}$, we have
		\[
		\|v\|^2_{L^{\alpha+2}} \leq C_1 \|v\|^2_{H^1} \leq C_2 H_\omega(v) = C_2 \|v\|^{\alpha+2}_{L^{\alpha+2}},
		\]
		for some $C_1, C_2>0$. This implies that
		\[
		\frac{\alpha}{2(\alpha+2)} \|v\|^{\alpha+2}_{L^{\alpha+2}} \geq \frac{\alpha}{2(\alpha+2)} \left( \frac{1}{C_2}\right)^{\frac{\alpha+2}{\alpha}}.
		\]
		Taking the infimun over $v \in H^1_{\text{rad}} \backslash \{0\}$, we obtain $d(\text{rad},\omega)>0$. 
	\end{proof}

	We now denote the set of all minimizers of $(\ref{minimizing problem})$ by
	\[
	\Mcal_{\text{rad},\omega}:= \left\{ v \in H^1_{\text{rad}} \backslash \{0\} \ : \ K_\omega(v) =0, \ S_\omega(v) = d(\text{rad},\omega) \right\}. 
	\]
	
	\begin{lemma} \label{lemma non empty M_omega}
		The set $\Mcal_{\emph{rad},\omega}$ is non-empty.
	\end{lemma}
	
	\begin{proof}
		Let $(v_n)_{n\geq 1}$ be a minimizing sequence of $d(\text{rad},\omega)$, i.e. $v_n \in H^1_{\text{rad}} \backslash \{0\}$, $K_\omega(v_n) =0$ and $S_\omega(v_n) \rightarrow d(\text{rad},\omega)$ as $n\rightarrow \infty$. Since $K_\omega(v_n) = 0$, we have $H_\omega(v_n) = \|v_n\|^{\alpha+2}_{L^{\alpha+2}}$ for any $n\geq 1$. Using $(\ref{expressions S_omega})$, the fact $S_\omega(v_n) \rightarrow d(\text{rad},\omega)$ as $n\rightarrow \infty$ implies that 
		\[
		\frac{\alpha}{2(\alpha+2)} H_\omega(v_n) = \frac{\alpha}{2(\alpha+2)} \|v_n\|^{\alpha+2}_{L^{\alpha+2}} \rightarrow d(\text{rad},\omega),
		\]
		as $n\rightarrow \infty$. We infer that there exists $C>0$ such that 
		\[
		H_\omega(v_n) \leq \frac{2(\alpha+2)}{\alpha} d(\text{rad},\omega) + C,
		\]
		for all $n\geq 1$. It follows from $(\ref{equivalent norms})$ that $(v_n)_{n\geq 1}$ is a bounded sequence in $H^1_{\text{rad}}$. Using the compact embedding $H^1_{\text{rad}} \hookrightarrow L^{\alpha+2}$, there exists $v_0 \in H^1_{\text{rad}}$ such that 
		\[
		v_n \rightharpoonup v_0 \text{ weakly in } H^1 \text{ and strongly in } L^{\alpha+2} \text{ as } n \rightarrow \infty. 
		\]
		Writting $v_n= v_0 + r_n$, where $r_n \rightharpoonup 0$ weakly in $H^1$ as $n\rightarrow \infty$. We have
		\[
		K_\omega(v_n) = H_\omega(v_n) - \|v_n\|^{\alpha+2}_{L^{\alpha+2}} = H_\omega(v_0) + H_\omega(r_n) - \|v_n\|^{\alpha+2}_{L^{\alpha+2}} + o_n(1),
		\]
		as $n\rightarrow \infty$. Here $o_n(1)$ means that $o_n(1) \rightarrow 0$ as $n\rightarrow \infty$. Since $K_\omega(v_n) =0$ and $H_\omega(r_n) \geq 0$ for all $n\geq 1$, we get
		\[
		H_\omega(v_0) \leq \|v_n\|^{\alpha+2}_{L^{\alpha+2}} + o_n(1),
		\]
		as $n\rightarrow \infty$. Taking the limit $n\rightarrow \infty$, we obtain
		\[
		H_\omega(v_0) \leq \frac{2(\alpha+2)}{\alpha} d(\text{rad},\omega).
		\]
		Since $v_n \rightarrow v_0$ strongly in $L^{\alpha+2}$, it follows that
		\[
		\|v_0\|^{\alpha+2}_{L^{\alpha+2}} =\lim_{n\rightarrow \infty} \|v_n\|^{\alpha+2}_{L^{\alpha+2}} = \frac{2(\alpha+2)}{\alpha} d(\text{rad},\omega).
		\]
		We thus get $K_\omega(v_0) \leq 0$. Now suppose that $K_\omega(v_0) <0$. We have for $\mu>0$,
		\[
		K_\omega(\mu v_0) = \mu^2 H_\omega(v_0) - \mu^{\alpha+2} \|v_0\|^{\alpha+2}_{L^{\alpha+2}}. 
		\]
		It is easy to see that the equation $K_\omega(\mu v_0)=0$ admits a unique non-zero solution
		\[
		\mu_0 = \left(\frac{H_\omega(v_0)}{\|v_0\|^{\alpha+2}_{L^{\alpha+2}}} \right)^{\frac{1}{\alpha}}.
		\]
		Since $K_\omega(v_0)<0$, we have $\mu_0 \in (0,1)$. By the definition of $d(\text{rad},\omega)$ and $(\ref{expressions S_omega})$, we get
		\begin{align*}
		d(\text{rad},\omega) \leq S_\omega(\mu_0 v_0) = \frac{\alpha}{2(\alpha+2)} H_\omega(\mu_0 v_0) &= \mu_0^2 \frac{\alpha}{2(\alpha+2)} H_\omega(v_0) \\
		&< \frac{\alpha}{2(\alpha+2)} H_\omega(v_0) \leq d(\text{rad},\omega),
		\end{align*}
		which is a contradiction. Therefore, $K_\omega(v_0)=0$. Moreover, 
		\[
		S_\omega(v_0) = \frac{\alpha}{2(\alpha+2)} \|v_0\|^{\alpha+2}_{L^{\alpha+2}} = d(\text{rad},\omega).
		\]
		This shows that $v_0$ is a minimizer of $d(\text{rad},\omega)$. The proof is complete.
	\end{proof}

	\begin{lemma} \label{lemma M_omega subset G_omega}
		$\Mcal_{\emph{rad},\omega} \subset \Gc_{\emph{rad},\omega}$.
	\end{lemma}
	
	\begin{proof}
		Let $\phi \in \Mcal_{\text{rad},\omega}$. Since $K_\omega(\phi) =0$, we have $H_\omega(\phi)=\|\phi\|^{\alpha+2}_{L^{\alpha+2}}$. Since $\phi$ is a minimizer of $d(\text{rad},\omega)$, there exists a Lagrange multiplier $\mu \in \R$ such that $S'_\omega(\phi) = \mu K'_\omega(\phi)$. We thus have
		\[
		0 = K_\omega(\phi) = \scal{S'_\omega(\phi), \phi} = \mu \scal{K'_\omega(\phi), \phi}.
		\]
		It is easy to see that
		\[
		K'_\omega(\phi) = -2\Delta \phi + 2\omega \phi -2 c |x|^{-2} \phi - (\alpha+2)|\phi|^\alpha \phi.
		\]
		Therefore, 
		\[
		\scal{K'_\omega(\phi),\phi}= 2 H_\omega(\phi) - (\alpha+2) \|\phi\|^{\alpha+2}_{L^{\alpha+2}} = -\alpha \|\phi\|^{\alpha+2}_{L^{\alpha+2}}<0.
		\]
		This implies that $\mu=0$, hence $S'_\omega(\phi) =0$. In particular, we have $\phi \in \Ac_{\text{rad},\omega}$. To prove $\phi \in \Gc_{\text{rad},\omega}$, it remains to show that $S_\omega(\phi) \leq S_\omega(v)$ for all $v \in \Ac_{\text{rad},\omega}$. To see this, let $v \in \Ac_{\text{rad},\omega}$. We have
		\[
		K_\omega(v) = \scal{S'_\omega(v),v} =0.
		\]
		By definition of $\Mcal_{\text{rad},\omega}$, we have $S_\omega(\phi) \leq S_\omega(v)$. The proof is complete. 
	\end{proof}

	\begin{lemma} \label{lemma G_omega subset M_omega}
		$\Gc_{\emph{rad},\omega} \subset \Mcal_{\emph{rad},\omega}$.
	\end{lemma}
	\begin{proof}
		Let $\phi \in \Gc_{\text{rad},\omega}$. Since $\Mcal_{\text{rad},\omega}$ is not empty, we take $\psi \in \Mcal_{\text{rad},\omega}$. By Lemma $\ref{lemma M_omega subset G_omega}$, $\psi \in \Gc_{\text{rad},\omega}$.
		In particular, $S_\omega(\phi) = S_\omega(\psi)$. Since $\psi \in \Mcal_{\text{rad},\omega}$, we get
		\[
		S_\omega(\phi) = S_\omega(\psi) = d(\text{rad},\omega).
		\]
		It remains to show that $K_\omega(\phi) =0$. Since $\phi \in \Ac_{\text{rad},\omega}$, we have $S'_\omega(\phi) =0$, hence $K_\omega(\phi) = \scal{S'_\omega(\phi),\phi} =0$. Therefore, $\phi \in \Mcal_{\text{rad},\omega}$ and the proof is complete.
	\end{proof}
	
	\noindent {\it Proof of Proposition $\ref{proposition existence radial ground states}$.} Proposition $\ref{proposition existence radial ground states}$ follows immediately from Lemmas $\ref{lemma non empty M_omega}$, $\ref{lemma M_omega subset G_omega}$ and $\ref{lemma G_omega subset M_omega}$.
	\defendproof

	\section{Instability of radial standing waves} \label{section instability}
	\setcounter{equation}{0}
	In this section, we give the proof of the instability of radial ground state standing waves given in Theorem $\ref{theorem instability}$. Let us start by recalling the local well-posedness in the energy space $H^1$ for $(\ref{inverse square NLS})$ proved by Okazawa-Suzuki-Yokota \cite{OkazawaSuzukiYokota}. 
	
	\begin{theorem}[Local well-posedness \cite{OkazawaSuzukiYokota}] \label{theorem local theory}
		Let $d\geq 3$, $c\ne 0$ be such that $c<\lambda(d)$ and $\frac{4}{d} <\alpha<\frac{4}{d-2}$. Then for any $u_0 \in H^1$, there exists $T \in (0, +\infty]$ and a maximal solution $u \in C([0,T), H^1)$ of $(\ref{inverse square NLS})$. The maximal time of existence satisfies either $T=+\infty$ or $T<+\infty$ and
		\[
		\lim_{t\uparrow T} \|\nabla u(t)\|_{L^2} =\infty.
		\]
		Moreover, the local solution enjoys the conservation of mass and energy
		\begin{align*}
		M(u(t)) &= \int |u(t,x)|^2 dx = M(u_0), \\
		E(u(t)) &= \frac{1}{2} \int |\nabla u(t,x)|^2 dx - \frac{c}{2} \int |x|^{-2} |u(t,x)|^2 dx - \frac{1}{\alpha+2} \int |u(t,x)|^{\alpha+2} dx \\
		&=E(u_0),
		\end{align*}
		for any $t\in [0,T)$.
	\end{theorem}
	
	We refer the reader to \cite[Proposition 5.1]{OkazawaSuzukiYokota} for the proof of the above result. Note that the existence of local solution is based on a refined energy method of the well-known energy method proposed by Cazenave \cite[Chapter 3]{Cazenave}. The uniqueness of local solutions follows from Strichartz estimates proved by Burq-Planchon-Stalker-Zadel \cite{BurPlaStaZad}.

	We next recall the so-called Pohozaev's identities for $(\ref{elliptic equation})$. We give the proof for the reader's convenience.
	\begin{lemma} \label{lemma pohozaev identities}
		Let $\omega>0$. If $\phi_\omega \in H^1$ is a solution to $(\ref{elliptic equation})$, then 
		\[
		\|\phi_\omega\|^2_{\dot{H}^1_c} + \omega \|\phi_\omega\|^2_{L^2} - \|\phi_\omega\|^{\alpha+2}_{L^{\alpha+2}} =0,
		\]
		and
		\[
		\left(1-\frac{d}{2}\right) \|\phi_\omega\|^2_{\dot{H}^1_c} -\frac{d\omega}{2} \|\phi_\omega\|^2_{L^2} + \frac{d}{\alpha+2} \|\phi_\omega\|^{\alpha+2}_{L^{\alpha+2}}.
		\]
	\end{lemma}
	\begin{proof}
		Multiplying both sides of $(\ref{elliptic equation})$ with $\phi_\omega$ and integrating over $\R^d$, we obtain easily the first identity. Let us prove the second identity. Due to the singularity of  the inverse-square potential at zero, we multiply both sides of $(\ref{elliptic equation})$ with $x \cdot \nabla \phi_\omega$ and integrate on $P(r,R):= \{ x \in \R^d \ : \ r \leq |x| \leq R\}$ for some $R>r>0$. We have 
		\begin{align*}
		-\int_{P(r,R)} \Delta \phi_\omega (x\cdot \nabla \phi_\omega) dx = \int_{P(r,R)} \nabla \phi_\omega \cdot \nabla(x\cdot \nabla \phi_\omega) dx &- \int_{\partial B_r} |\nabla \phi_\omega|^2 (x\cdot \n_1) dS \\
		&- \int_{\partial B_R} |\nabla \phi_\omega|^2 (x\cdot \n_2) dS,
		\end{align*} 
		where $\n_1= -\frac{x}{r}$ is the unit inward normal at $x \in \partial B_r$ and $\n_2 = \frac{x}{R}$ is the unit outward normal at $x \in \partial B_R$. We also have
		\begin{align*}
		\int_{P(r,R)} \nabla \phi_\omega \cdot \nabla(x\cdot \nabla \phi_\omega) dx = \left(1-\frac{d}{2}\right) \int_{P(r,R)} |\nabla \phi_\omega|^2 dx &+ \frac{1}{2} \int_{\partial B_r} |\nabla \phi_\omega|^2 (x\cdot \n_1)dS \\
		&+ \frac{1}{2} \int_{\partial B_R} |\nabla \phi_\omega|^2 (x\cdot \n_2)dS.
		\end{align*}
		Thus,
		\begin{align*}
		-\int_{P(r,R)} \Delta \phi_\omega (x\cdot \nabla \phi_\omega) dx = \left(1-\frac{d}{2}\right) \int_{P(r,R)} |\nabla \phi_\omega|^2 dx &- \frac{1}{2} \int_{\partial B_r} |\nabla \phi_\omega|^2 (x\cdot \n_1) dS \\
		&-\frac{1}{2} \int_{\partial B_R} |\nabla \phi_\omega|^2 (x\cdot \n_2) dS.
		\end{align*}
		Similarly,
		\begin{align*}
		\omega\int_{P(r,R)} \phi_\omega(x\cdot \nabla \phi_\omega) dx = -\frac{d\omega}{2} \int_{P(r,R)} |\phi_\omega|^2 dx &+ \frac{\omega}{2} \int_{\partial B_r}|\phi_\omega|^2 (x\cdot \n_1) dS \\
		&+\frac{\omega}{2} \int_{\partial B_R}|\phi_\omega|^2 (x\cdot \n_2) dS,
		\end{align*}
		and
		\begin{align*}
		-c \int_{P(r,R)} |x|^{-2} \phi_\omega(x\cdot \nabla \phi_\omega) dx &= -c\left(1-\frac{d}{2}\right) \int_{P(r,R)} |x|^{-2} |\phi_\omega|^2 dx \\
		&\mathrel{\phantom{=}}-\frac{c}{2} \int_{\partial B_r} |x|^{-2} |\phi_\omega|^2 (x\cdot \n_1) dS \\
		&\mathrel{\phantom{=}}-\frac{c}{2} \int_{\partial B_R} |x|^{-2} |\phi_\omega|^2 (x\cdot \n_2) dS,
		\end{align*}
		and finally
		\begin{align*}
		-\int_{P(r,R)} |\phi_\omega|^\alpha \phi_\omega (x \cdot \nabla \phi_\omega) dx &= \frac{d}{\alpha+2} \int_{P(r,R)} |\phi_\omega|^{\alpha+2} dx \\
		&\mathrel{\phantom{=}}-\frac{1}{\alpha+2} \int_{\partial B_r} |\phi_\omega|^{\alpha+2} (x\cdot \n_1) dS \\
		&\mathrel{\phantom{=}}-\frac{1}{\alpha+2} \int_{\partial B_R} |\phi_\omega|^{\alpha+2} (x\cdot \n_2) dS.
		\end{align*}
		Adding the above identities, we get
		\begin{multline}
		\left(1-\frac{d}{2}\right) \left[\int_{P(r,R)} |\nabla \phi_\omega|^2 dx - c \int_{P(r,R)} |x|^{-2} |\phi_\omega|^2 dx\right] -\frac{d\omega}{2} \int_{P(r,R)}|\phi_\omega|^2 dx \\
		+ \frac{d}{\alpha+2} \int_{P(r,R)} |\phi_\omega|^{\alpha+2} dx = I_1(r) + I_2(R), \label{pohozaev proof}
		\end{multline}
		where
		\begin{align*}
		I_1(r) &= \frac{1}{2}\int_{\partial B_r} |\nabla \phi_\omega|^2 (x\cdot \n_1) dS - \frac{\omega}{2} \int_{\partial B_r} |\phi_\omega|^2 (x\cdot\n_1) dS  \\
		&\mathrel{\phantom{=}} +\frac{c}{2}\int_{\partial B_r} |x|^{-2} |\phi_\omega|^2 (x\cdot \n_1) dS +\frac{1}{\alpha+2} \int_{\partial B_r} |\phi_\omega|^{\alpha+2} (x\cdot \n_1) dS \\
		&=-r\left( \int_{\partial B_r} \frac{1}{2} |\nabla \phi_\omega|^2 -\frac{\omega}{2}|\phi_\omega|^2 + \frac{c}{2} |x|^{-2} |\phi_\omega|^2 +\frac{1}{\alpha+2} |\phi_\omega|^{\alpha+2} dS \right),
		\end{align*}
		and 
		\begin{align*}
		I_2(R) &= \frac{1}{2}\int_{\partial B_R} |\nabla \phi_\omega|^2 (x\cdot \n_2) dS - \frac{\omega}{2} \int_{\partial B_R} |\phi_\omega|^2 (x\cdot\n_2) dS  \\
		&\mathrel{\phantom{=}} +\frac{c}{2}\int_{\partial B_R} |x|^{-2} |\phi_\omega|^2 (x\cdot \n_2) dS +\frac{1}{\alpha+2} \int_{\partial B_R} |\phi_\omega|^{\alpha+2} (x\cdot \n_2) dS \\
		&=R\left( \int_{\partial B_R} \frac{1}{2} |\nabla \phi_\omega|^2 -\frac{\omega}{2}|\phi_\omega|^2 + \frac{c}{2} |x|^{-2} |\phi_\omega|^2 +\frac{1}{\alpha+2} |\phi_\omega|^{\alpha+2} dS \right).
		\end{align*}
		Denote 
		\[
		A(\phi_\omega) = \frac{1}{2} |\nabla \phi_\omega|^2 -\frac{\omega}{2}|\phi_\omega|^2 + \frac{c}{2} |x|^{-2} |\phi_\omega|^2 +\frac{1}{\alpha+2} |\phi_\omega|^{\alpha+2}.
		\]
		We have
		\begin{align}
		\int_{B} A(\phi_\omega) dx = \int_0^1 \int_{\partial B_r} A(\phi_\omega) dS dr <\infty, \label{finite term}
		\end{align}
		where $B$ is the unit ball in $\R^d$. Hence, there exists a sequence $r_n \rightarrow 0$ such that 
		\[
		r_n \int_{\partial B_{r_n}} A(\phi_\omega) dS \rightarrow 0 \quad \text{as } n\rightarrow \infty.
		\]
		Indeed, if 
		\[
		\liminf_{r\rightarrow 0} r \int_{\partial B_r} A(\phi_\omega) dS = c >0, 
		\]
		then 
		\[
		\int_{\partial B_r} A(\phi_\omega) dS
		\]
		would not be in $L^1(0,1)$, which contradicts to $(\ref{finite term})$. On the other hand, since
		\[
		\int_{\R^d} A(\phi_\omega) dx = \int_0^{+\infty} \int_{\partial B_R} A(\phi_\omega) dS dR <\infty,
		\]
		there exists a sequence $R_n \rightarrow +\infty$ such that 
		\[
		R_n\int_{\partial B_R} A(\phi_\omega) dS \rightarrow 0 \quad \text{as } n\rightarrow \infty.
		\]
		This implies that $I_1(r_n) \rightarrow 0$ and $I_2(R_n) \rightarrow 0$ as $n\rightarrow \infty$. Now substituting $r$ by $r_n$ and $R$ by $R_n$ in $(\ref{pohozaev proof})$ and taking $n\rightarrow \infty$, we obtain the second identity. The proof is complete.
	\end{proof}

	Throughout this section, we denote the functional
	\[
	Q(v):= \|v\|^2_{\dot{H}^1_c} -\frac{d\alpha}{2(\alpha+2)} \|v\|^{\alpha+2}_{L^{\alpha+2}}.
	\]
	Note that if we take 
	\begin{align}
	v^\lambda(x):= \lambda^{\frac{d}{2}} v(\lambda x), \label{scaling}   
	\end{align} 
	then we have
	\begin{align}
	\begin{aligned}
	\|v^\lambda\|_{L^2} &= \|v\|_{L^2}, & \|\nabla v^\lambda\|_{L^2} &= \lambda \|\nabla v\|_{L^2}, \\
	\||x|^{-1} v^\lambda \|_{L^2} &= \lambda \||x|^{-1} v\|_{L^2}, & \|v^\lambda\|_{L^{\alpha+2}} &= \lambda^{\frac{d\alpha}{2(\alpha+2)}} \|v\|_{L^{\alpha+2}}.
	\end{aligned}
	\label{scaling examples}
	\end{align}
	Thus,
	\[
	S_\omega(v^\lambda) = \frac{\lambda^2}{2} \|v\|^2_{\dot{H}^1_c} + \frac{\omega}{2} \|v\|^2_{L^2} - \frac{\lambda^{\frac{d\alpha}{2}}}{\alpha+2} \|v\|^{\alpha+2}_{L^{\alpha+2}},
	\]
	and
	\[
	Q(v) = \left. \partial_\lambda S_\omega(v^\lambda)\right|_{\lambda=1}.
	\]
	\begin{lemma} \label{lemma S_ome phi_ome}
		Let $d\geq 3, c \ne 0$ be such that $c<\lambda(d)$, $\frac{4}{d}<\alpha < \frac{4}{d-2}$ and $\omega>0$. Let $\phi_\omega \in \mathcal{G}_{\emph{rad},\omega}$. Then 
		\[
		S_\omega(\phi_\omega) = \inf \{S_\omega(v) \ : \ v \in H^1_{\text{rad}} \backslash \{0\}, Q(v)=0 \}.
		\]
	\end{lemma}
	\begin{proof}
		Let $d_n:= \inf \{S_\omega(v) \ : \ v \in H^1_{\text{rad}} \backslash \{0\}, Q(v)=0 \}$. Thanks to the Pohozaev's identities, it is easy to check that $S_\omega(\phi_\omega)= Q(\phi_\omega)=0$. By the definition of $d_n$,
		\begin{align}
		S_\omega(\phi_\omega) \geq d_n. \label{inequality 1}
		\end{align}
		We now consider $v \in H^1_{\text{rad}} \backslash \{0\}$ be such that $Q(v)=0$. If $K_\omega(v)=0$, then by Proposition $\ref{proposition existence radial ground states}$, $S_\omega(v) \geq S_\omega(\phi_\omega)$. Assume that $K_\omega(v) \ne 0$. Let $v^\lambda$ be as in $(\ref{scaling})$. We have
		\[
		K_\omega(v^\lambda)=\lambda^2 \|v\|^2_{\dot{H}^1_c} + \omega \|v\|^2_{L^2} - \lambda^{\frac{d\alpha}{2}} \|v\|^{\alpha+2}_{L^{\alpha+2}}.
		\]
		We see that $\lim_{\lambda\rightarrow 0} K_\omega(v^\lambda)= \omega \|v\|^2_{L^2}>0$. Since $\frac{d\alpha}{2}>2$, we have $\lim_{\lambda \rightarrow +\infty} K_\omega(v^\lambda) = -\infty$. Thus, there exists $\lambda_0 >0$ such that $K_\omega(v^{\lambda_0})=0$. By Proposition $\ref{proposition existence radial ground states}$, we get $S_\omega(v^{\lambda_0}) \geq S_\omega(\phi_\omega)$. On the other hand, a direct computation shows that
		\begin{align*}
		\partial_\lambda S_\omega(v^\lambda)&= \lambda \|v\|^2_{\dot{H}^1_c} -\frac{d\alpha}{2(\alpha+2)} \lambda^{\frac{d\alpha}{2}-1} \|v\|^{\alpha+2}_{L^{\alpha+2}} \\
		&= \lambda \left( \|v\|^2_{\dot{H}^1_c}- \frac{d\alpha}{2(\alpha+2)} \lambda^{\frac{d\alpha}{2}-2} \|v\|^{\alpha+2}_{L^{\alpha+2}} \right).
		\end{align*}
		The equation $\partial_\lambda S_\omega(v^\lambda)= 0$ admits a unique non-zero solution 
		\[
		\lambda_1 = \left( \frac{\|u\|^2_{\dot{H}^1_c}}{\frac{d\alpha}{2(\alpha+2)} \|v\|^{\alpha+2}_{L^{\alpha+2}} } \right)^{\frac{2}{d\alpha-4}}
		\]
		which is equal to 1 since $Q(v)=0$. It follows that $\partial_\lambda S_\omega(v^\lambda)>0$ if $\lambda \in (0,1)$ and $\partial_\lambda S_\omega(v^\lambda)<0$ if $\lambda \in (1,\infty)$. In particular, we get $S_\omega(v^\lambda) < S_\omega(v)$ for any $\lambda>0$ and $\lambda \ne 1$. Since $\lambda_0>0$, it follows that $S_\omega(v^{\lambda_0}) \leq S_\omega(v)$. This implies that $S_\omega(v) \geq S_\omega(\phi_\omega)$ for any $v \in H^1_{\text{rad}} \backslash \{0\}$, $Q(v)=0$. Taking the infimum, we obtain
		\begin{align}
		S_\omega(\phi_\omega) \leq d_n. \label{inequality 2}
		\end{align}
		Combining $(\ref{inequality 1})$ and $(\ref{inequality 2})$, we prove the result.
	\end{proof}
	
	Let $\phi_\omega \in \mathcal{G}_{\text{rad},\omega}$. We denote
	\[
	\mathcal{B}_{\text{rad},\omega} := \{ v \in H^1_{\text{rad}} \backslash \{0\} \ : \ S_\omega(v) < S_\omega(\phi_\omega), Q(v) <0 \}.
	\]
	\begin{lemma} 
		Let $d\geq 3, c \ne 0$ be such that $c<\lambda(d)$, $\frac{4}{d}<\alpha <\frac{4}{d-2}$ and $\omega>0$. Let $\phi_\omega \in \mathcal{G}_{\text{rad},\omega}$. Then $\mathcal{B}_{\text{rad},\omega}$ is invariant under the flow of $(\ref{inverse square NLS})$, that is, if $u_0 \in \mathcal{B}_{\text{rad},\omega}$, then the corresponding solution $u(t)$ to $(\ref{inverse square NLS})$ with $u(0) = u_0$ satisfies $u(t) \in \mathcal{B}_{\text{rad},\omega}$ for any $t\in [0,T)$.
	\end{lemma}
	\begin{proof}
		Let $u_0 \in \mathcal{B}_{\text{rad},\omega}$. By the conservation of mass and energy,
		\begin{align}
		S_\omega(u(t)) = S_\omega(u_0) < S_\omega(\phi_\omega), \quad \forall t\in [0,T). \label{invariant property}
		\end{align}
		It remains to show that $Q(u(t))<0$ for any $t\in [0,T)$. Suppose that there exists $t_0 \in [0,T)$ such that $Q(u(t_0)) \geq 0$. By the continuity of $t\mapsto Q(u(t))$, there exists $t_1 \in (0, t_0]$ such that $Q(u(t_1))=0$. By Lemma $\ref{lemma S_ome phi_ome}$, $S_\omega(u(t_1)) \geq S_\omega(\phi_\omega)$ which contradicts to $(\ref{invariant property})$. 
	\end{proof}

	\begin{lemma} \label{lemma key estimate}
		Let $d\geq 3, c\ne 0$ be such that $c<\lambda(d)$, $\frac{4}{d}<\alpha<\frac{4}{d-2}$ and $\omega>0$. Let $\phi_\omega \in \mathcal{G}_{\text{rad},\omega}$. If $v \in \mathcal{B}_{\text{rad},\omega}$, then 
		\[
		Q(v) \leq 2 ( S_\omega(v) - S_\omega(\phi_\omega)).
		\]
	\end{lemma}
	\begin{proof}
		Let $v^\lambda$ be as in $(\ref{scaling})$. Set $g(\lambda):= S_\omega(v^\lambda)$. We have
		\begin{align*}
		g(\lambda) &= \frac{\lambda^2}{2} \|v\|^2_{\dot{H}^1_c} + \frac{\omega}{2} \|v\|^2_{L^2} - \frac{\lambda^{\frac{d\alpha}{2}}}{\alpha+2} \|v\|^{\alpha+2}_{L^{\alpha+2}}, \\
		g'(\lambda)&= \lambda \|v\|^2_{\dot{H}^1_c} -\frac{d\alpha}{2(\alpha+2)} \lambda^{\frac{d\alpha}{2}-1} \|v\|^{\alpha+2}_{L^{\alpha+2}} = \frac{Q(v^\lambda)}{\lambda}, \\
		\end{align*}
		and
		\begin{align*}
		(\lambda g'(\lambda))' &= 2 \lambda \|v\|^2_{\dot{H}^1_c} - \frac{d^2\alpha^2}{4(\alpha+2)} \lambda^{\frac{d\alpha}{2}-1} \|v\|^{\alpha+2}_{L^{\alpha+2}} \\
		&= 2 \left( \lambda \|v\|^2_{\dot{H}^1_c} - \frac{d\alpha}{2(\alpha+2)} \lambda^{\frac{d\alpha}{2}-1} \|v\|^{\alpha+2}_{L^{\alpha+2}} \right) - \frac{d\alpha (d\alpha-4)}{4(\alpha+2)} \lambda^{\frac{d\alpha}{2}-1} \|v\|^{\alpha+2}_{L^{\alpha+2}} \\
		&= 2 g'(\lambda) - \frac{d\alpha (d\alpha-4)}{4(\alpha+2)} \lambda^{\frac{d\alpha}{2}-1} \|v\|^{\alpha+2}_{L^{\alpha+2}}.
		\end{align*}
		Since $d\alpha>4$, we see that
		\begin{align}
		(\lambda g'(\lambda))' \leq 2 g'(\lambda), \quad \forall \lambda>0. \label{integration}
		\end{align}
		Since $Q(v) <0$, the equation $\partial_\lambda S_\omega(v^\lambda) =0$ admits a unique non-zero solution $\lambda_0 \in (0,1)$. Taking the integration over $\lambda_0$ and 1 and note that $Q(v^{\lambda_0}) = \lambda_0 \left.\left( \partial_\lambda S_\omega(v^\lambda) \right)\right|_{\lambda =\lambda_0} =0$, we get
		\[
		Q(v) - Q(v^{\lambda_0}) \leq 2 (S_\omega(v) - S_\omega(v^{\lambda_0})) \leq 2 (S_\omega(v) - S_\omega(\phi_\omega)).
		\]
		Here, the last inequality comes from the fact $Q(v^{\lambda_0})=0$. The proof is complete.
	\end{proof}

	The key ingredient in showing the strong instability of radial standing waves is to use localized virial estimates to establish the finite time blowup. Let us recall localized virial estimates related to $(\ref{inverse square NLS})$. Let $\theta: [0,\infty) \rightarrow [0,\infty)$ be such that
	\[
	\theta(r) = \left\{ 
	\begin{array}{cl}
	r^2 &\text{if } 0\leq r\leq 1, \\
	\text{const.} &\text{if } r \geq 2,
	\end{array}
	\right.
	\quad \text{and} \quad \theta''(r) \leq 2 \text{ for } r\geq 0.
	\]
	The precise constant here is not important. For $R>1$, we define the radial function 
	\begin{align}
	\varphi_R(x) = \varphi_R(r) := R^2 \theta(r/R), \quad r=|x|. \label{define varphi_R}
	\end{align}
	We define the virial potential by
	\begin{align}
	V_{\varphi_R}(t) := \int \varphi_R(x) |u(t,x)|^2 dx. \label{define virial potential}
	\end{align}
	\begin{lemma}[Radial virial estimate \cite{Dinh-inverse}] \label{lemma radial virial estimate}
		Let $d\geq 3$, $c \ne 0$ be such that $c<\lambda(d)$, $\frac{4}{d} <\alpha<\frac{4}{d-2}$, $R>1$ and $\varphi_R$ be as in $(\ref{define varphi_R})$. Let $u: I\times \R^d \rightarrow \C$ be a radial solution to $(\ref{inverse square NLS})$. Then for any $t \in I$, 
		\begin{align}
		\frac{d^2}{dt^2} V_{\varphi_R}(t) &\leq 8 \|u(t)\|^2_{\dot{H}^1_c} - \frac{4d\alpha}{\alpha+2} \|u(t)\|^{\alpha+2}_{L^{\alpha+2}} + O \left( R^{-2} + R^{-\frac{(d-1)\alpha}{2}} \|u(t)\|^{\frac{\alpha}{2}}_{\dot{H}^1_c} \right) \label{radial virial estimate 1} \\
		&= 8 Q(u(t)) + O \left( R^{-2} + R^{-\frac{(d-1)\alpha}{2}} \|u(t)\|^{\frac{\alpha}{2}}_{\dot{H}^1_c} \right)  \label{radial virial estimate 2} \\
		&=4d\alpha E(u(t)) - 2(d\alpha-4) \|u(t)\|^2_{\dot{H}^1_c} + O \left( R^{-2} + R^{-\frac{(d-1)\alpha}{2}} \|u(t)\|^{\frac{\alpha}{2}}_{\dot{H}^1_c} \right). \label{radial virial estimate 3}
		\end{align}
		The implicit constant depends only on $\|u_0\|_{L^2}, d$ and $\alpha$. Here $A=O(B)$ means there exists a constant $C>0$ such that $A=CB$. 
	\end{lemma}
	We refer the reader to \cite[Lemma 5.4]{Dinh-inverse} for the proof of the above result. 
	
	We are now able to prove our main result.
	
	\noindent {\it Proof of Theorem $\ref{theorem instability}$.}
	Let $\eps>0$, $\omega>0$ and $\phi_\omega \in \Gc_{\text{rad},\omega}$. Since $\phi^\lambda_\omega \rightarrow \phi_\omega$ in $H^1$ as $\lambda \rightarrow 1$, there exists $\lambda_0>1$ such that $\|\phi_\omega - \phi^{\lambda_0}_\omega \|_{H^1} <\eps$. By decreasing $\lambda_0$ if necessary, we claim that $\phi^{\lambda_0}_\omega \in \mathcal{B}_{\text{rad},\omega}$. To see this, we first notice that $Q(\phi_\omega)=0$. This fact follows from the Pohozaev's identities related to $(\ref{elliptic equation})$ given in Lemma $\ref{lemma pohozaev identities}$:
	\begin{align}
	\omega \|\phi_\omega\|^2_{L^2} = \frac{4-(d-2)\alpha}{2(\alpha+2)} \|\phi_\omega\|^{\alpha+2}_{L^{\alpha+2}} = \frac{4-(d-2)\alpha}{d\alpha} \|\phi_\omega\|^2_{\dot{H}^1_c}. \label{pohozaev identities}
	\end{align}
	
	On the other hand, a direct computation shows
	\begin{align*}
	S_\omega(\phi^{\lambda}_\omega) &:= \frac{\lambda^2}{2} \|\phi_\omega\|^2_{\dot{H}^1_c} + \frac{\omega}{2} \|\phi_\omega\|^2_{L^2} - \frac{\lambda^{\frac{d\alpha}{2}}}{\alpha+2} \|\phi_\omega\|^{\alpha+2}_{L^{\alpha+2}}, \\
	\partial_\lambda S_\omega(\phi^\lambda_\omega) &:= \lambda \|\phi_\omega\|^2_{\dot{H}^1_c} - \frac{d\alpha}{2(\alpha+2)} \lambda^{\frac{d\alpha-2}{2}} \|\phi_\omega\|^{\alpha+2}_{L^{\alpha+2}} = \frac{Q(\phi^\lambda_\omega)}{\lambda}.
	\end{align*}
	It is easy to see that the equation $\partial_\lambda S_\omega(\phi^\lambda_\omega) =0$ has a unique non-zero solution 
	\[
	\left( \frac{\|\phi_\omega\|^2_{\dot{H}^1_c} }{ \frac{d\alpha}{2(\alpha+2)} \|\phi_\omega\|^{\alpha+2}_{L^{\alpha+2}} } \right)^{\frac{2}{d\alpha-4}} =1.
	\]
	The last inequality comes from the fact $Q(\phi_\omega) =0$. This implies in particular that
	\[
	\left\{
	\begin{array}{c l}
	\partial_\lambda S_\omega(\phi^\lambda_\omega)>0 &\text{if } \lambda \in (0,1), \\
	\partial_\lambda S_\omega(\phi^\lambda_\omega)<0 &\text{if } \lambda \in (1,\infty),
	\end{array}
	\right.
	\]
	from which we get $S_\omega(\phi^\lambda_\omega) < S_\omega(\phi_\omega)$ for any $\lambda>0, \lambda \ne 1$. Since $Q(\phi^\lambda_\omega) = \lambda  \partial_\lambda S_\omega (\phi^\lambda_\omega)$, we also have 
	\[
	\left\{
	\begin{array}{c l}
	Q(\phi^\lambda_\omega)>0 &\text{if } \lambda \in (0,1), \\
	Q(\phi^\lambda_\omega)<0 &\text{if } \lambda \in (1,\infty).
	\end{array}
	\right.
	\]
	As an application of the above argument, we have
	\[
	S_\omega(\phi^{\lambda_0}_\omega) <S_\omega(\phi_\omega), \quad Q(\phi^{\lambda_0}_\omega) <0.
	\]
	This shows that $\phi^{\lambda_0}_\omega \in \mathcal{B}_{\text{rad},\omega}$ and the claim follows. 
	
	By Theorem $\ref{theorem local theory}$, there exists a unique solution $u \in C([0,T),H^1)$ to $(\ref{inverse square NLS})$ with initial data $u(0)=u_0= \phi^{\lambda_0}_\omega$, where $T>0$ is the maximal existence time. Since $u_0= \phi^{\lambda_0}_\omega$ is radial, it is well-known that the corresponding solution is also radial. The rest of this note is to show that $u$ blows up in finite time. It is done by several steps.
	
	\noindent {\bf Step 1.} We claim that there exists $a>0$ such that $Q(u(t)) \leq -a$ for any $t \in [0,T)$. Indeed, since $\mathcal{B}_{\text{rad},\omega}$ is invariant under the flow of $(\ref{inverse square NLS})$, we see that $u(t) \in \mathcal{B}_{\text{rad},\omega}$ for any $t\in [0,T)$. By Lemma $\ref{lemma key estimate}$, we get
	\[
	Q(u(t)) \leq 2(S_\omega(u(t)) - S_\omega(\phi_\omega)) = 2 (S_\omega(\phi^{\lambda_0}_\omega) - S_\omega(\phi_\omega)).
	\]
	This proves the claim with $a= 2(S_\omega(\phi_\omega)- S_\omega(\phi^{\lambda_0}_\omega))>0$. 
	
	\noindent {\bf Step 2.} We next claim that there exists $b>0$ such that 
	\begin{align}
	\frac{d^2}{dt^2} V_{\varphi_R}(t) \leq -b, \label{virial potential bound}
	\end{align}
	for any $t \in [0,T)$, where $V_{\varphi_R}(t)$ is as in $(\ref{define virial potential})$. Indeed, since the solution $u(t)$ is radial, we apply Lemma $\ref{lemma radial virial estimate}$ to have
	\[
	\frac{d^2}{dt^2} V_{\varphi_R}(t) \leq 4d\alpha E(u(t)) - 2(d\alpha-4) \|u(t)\|^2_{\dot{H}^1_c} + O\left( R^{-2} + R^{-\frac{(d-1)\alpha}{2}} \|u(t)\|^{\frac{\alpha}{2}}_{\dot{H}^1_c} \right),
	\]
	for any $t\in [0,T)$ and any $R>1$. The Young inequality implies for any $\eps>0$,
	\[
	R^{-\frac{(d-1)\alpha}{2}} \|u(t)\|^{\frac{\alpha}{2}}_{\dot{H}^1_c} \lesssim \eps \|u(t)\|^2_{\dot{H}^1_c} + \eps^{-\frac{\alpha}{4-\alpha}} R^{-\frac{2(d-1)\alpha}{4-\alpha}}.
	\]
	Note that in our consideration, we always have $0<\alpha<4$. 		We thus get
	\[
	\frac{d^2}{dt^2} V_{\varphi_R}(t) \leq 4d\alpha E(u(t)) - 2(d\alpha-4) \|u(t)\|^2_{\dot{H}^1_c} + C\eps \|u(t)\|^2_{\dot{H}^1_c} + O\left( R^{-2} + \eps^{-\frac{\alpha}{4-\alpha}} R^{-\frac{2(d-1)\alpha}{4-\alpha}} \right),
	\]
	for any $t\in [0,T)$, any $R>1$, any $\eps>0$ and some constant $C>0$.
	
	To see $(\ref{virial potential bound})$, we follow the argument of Bonheure-Cast\'eras-Gou-Jeanjean \cite{BonheureCasterasGouJeanjean}. Fix $t \in [0,T)$ and denote
	\[
	\mu:= \frac{4d\alpha |E(u_0)| +2}{d\alpha-4}.
	\]
	We consider two cases.
	
	\noindent {\bf Case 1.} 
	\[
	\|u(t)\|^2_{\dot{H}^1_c} \leq \mu.
	\]
	Since $4d\alpha E(u(t)) - 2(d\alpha-4) \|u(t)\|^2_{\dot{H}^1_c} =8Q(u(t)) \leq -8a$ for any $t\in [0,T)$, we have
	\[
	\frac{d^2}{dt^2} V_{\varphi_R}(t) \leq -8a + C\eps \mu + O\left( R^{-2} + \eps^{-\frac{\alpha}{4-\alpha}} R^{-\frac{2(d-1)\alpha}{4-\alpha}} \right).
	\]
	By choosing $\eps>0$ small enough and $R>1$ large enough depending on $\eps$, we see that
	\[
	\frac{d^2}{dt^2} V_{\varphi_R}(t) \leq -4a.
	\]
	
	\noindent {\bf Case 2.}
	\[
	\|u(t)\|^2_{\dot{H}^1_c} >\mu.
	\]
	In this case, we have
	\[
	4d\alpha E(u_0) - 2(d\alpha-4) \|u(t)\|^2_{\dot{H}^1_c} < -2 -(d\alpha -4) \|u(t)\|^2_{\dot{H}^1_c}.
	\]
	Thus,
	\[
	\frac{d^2}{dt^2} V_{\varphi_R}(t) \leq -2 -(d\alpha-4) \|u(t)\|^2_{\dot{H}^1_c} + C\eps \|u(t)\|^2_{\dot{H}^1_c} + O \left( R^{-2} + \eps^{-\frac{\alpha}{4-\alpha}} R^{-\frac{2(d-1)\alpha}{4-\alpha}} \right).
	\]
	Since $d\alpha-4 >0$, we choose $\eps>0$ small enough so that 
	\[
	d\alpha-4 - C\eps \geq 0.
	\]
	This implies that
	\[
	\frac{d^2}{dt^2} V_{\varphi_R}(t) \leq -2 + O\left( R^{-2} + \eps^{-\frac{\alpha}{4-\alpha}} R^{-\frac{2(d-1)\alpha}{4-\alpha}} \right).
	\]
	We next choose $R>1$ large enough depending on $\eps$ so that 
	\[
	\frac{d^2}{dt^2} V_{\varphi_R}(t) \leq -1.
	\]
	Note that in both cases, the choices of $\eps>0$ and $R>1$ are independent of $t$. Therefore, the claim follows with $b= \min\{4a, 1\}>0$. 
	
	\noindent {\bf Step 3.} By Step 2, the solution $u(t)$ satisfies
	\[
	\frac{d^2}{dt^2} V_{\varphi_R}(t) \leq -b <0,
	\]
	for any $t\in [0,T)$. The convexity argument of Glassey (see e.g. \cite{Glassey}) implies that the solution blows up in finite time. The proof is complete.	
	\defendproof

	\section*{Acknowledgments}
	V. D. Dinh would like to express his deep gratitude to his wife-Uyen Cong for her encouragement and support. The authors would like to thank the reviewers for their helpful comments and suggestions.

\end{document}